\documentclass[10pt]{article}
\usepackage{indentfirst, latexsym, amsmath,bm,amssymb}
\usepackage{tikz}
\usepackage{color,soul}
\usepackage{amsmath}
\usepackage{amssymb}
\usepackage{mathrsfs}
\usepackage{mathtools}
\usepackage{stmaryrd}
\usepackage{bbold}

\allowdisplaybreaks[4]

\begin{document}
\newtheorem{fact}{Fact}
\newtheorem{theorem}{Theorem}[section]
\newtheorem{claim}{Claim}[section]
\newtheorem{corollary}[theorem]{Corollary}
\newtheorem{definition}[theorem]{Definition}
\newtheorem{conjecture}[theorem]{Conjecture}
\newtheorem{question}[theorem]{Question}
\newtheorem{lemma}[theorem]{Lemma}
\newtheorem{remark}[theorem]{Remark}
\newtheorem{proposition}[theorem]{Proposition}
\newtheorem{example}[theorem]{Example}
\newenvironment{proof}{\noindent {\bf
Proof.}}{\rule{3mm}{3mm}\par\medskip}
\newenvironment{proofclaim}{\noindent {\bf
Proof of  Claim}}{\rule{2mm}{2mm}\par\medskip}
\newenvironment{proofof}{\noindent {\bf
Proof of  Theorem}}{\rule{3mm}{3mm}\par\medskip}
\newcommand{\pp}{{\it p.}}
\newcommand{\de}{\em}

\title{  {The signless Laplacian spectral Tur\'{a}n problems for disjoint triangles
}\thanks{
Ming-Zhu Chen is partly supported by the National Natural Science Foundation of China (No. 12461063),
Ya-Lei Jin is partly supported by the National Natural Science Foundation of China (Nos. 12371349 and 12471331),
Peng-Li Zhang is partly supported by the National Natural Science Foundation of China (No. 12501493).
The authors   are listed in alphabetical order.
E-mail addresses: mzchen@hainanu.edu.cn (M.-Z. Chen), yaleijin@shnu.edu.cn ($^\dag$Y.-L. Jin, corresponding author), zpengli0626@163.com ($^\dag$P.-L. Zhang, corresponding author).}}

\author{
Ming-Zhu Chen$^{a}$,
Ya-Lei Jin$^{b\dagger}$,
Peng-Li Zhang$^{c\dagger}$\\[0.5em]
{\small $^{a}$ School of Mathematics and Statistics,}\\
{\small Key Laboratory of Engineering Modeling and Statistical Computation of Hainan Province,}\\
{\small Hainan University, Haikou 570228, P.R. China}\\[0.3em]
{\small $^{b}$ Department of Mathematics, Shanghai Normal University,}\\
{\small Shanghai 200234, P.R. China}\\[0.3em]
{\small $^{c}$ School of Statistics and Data Science,}\\
{\small Shanghai University of International Business and Economics,}\\
{\small Shanghai 201620, P.R. China}
}

\date{}

\maketitle

\vspace{-0.5cm}

\begin{abstract}
The classical   Tur\'{a}n-type problem for graphs forbidding $k$ disjoint triangles has been resolved firstly by Erd\H{o}s [Arch. Math. 13 (1962) 222-227] for  graphs of  sufficiently large order  and improved by  Moon [Canad. J. Math. 20 (1968) 95-102],
who established that the unique edge extremal graph is $K_{k-1}\vee T_2(n-k+1)$  for  $n> \frac{19k-11}{2}$. 
In this paper,
we  determine the maximum signless Laplacian spectral radius among all $kK_3$-free graphs of order $n$ and characterize the unique  signless Laplacian spectral extremal graph $K_{k-1}\vee  T_{2}(n-k+1)$ for all  $k\ge 2$ and sufficiently large $n$.
\end{abstract}


\section{Introduction}

We say that a graph is $F$-free if it does not contain $F$ as a subgraph.  One of the most central topics in extremal combinatorics, the Tur\'{a}n-type problem, asks for  the maximum  number of edges in    an $n$-vertex $F$-free graph. This maximum number of edges is called the {\it Tur\'{a}n number} of $F$, denoted by $ex(n,F).$ An $n$-vertex $F$-free graph with $ex(n,F)$ edges is called an {\it edge extremal graph} for $F$.   In 1907, Mantel \cite{Mantel1907} characterized the unique edge
extremal graph for  $K_3$. In 1941, Tur\'{a}n~\cite{Turan 1941}  extended the result of Mantel by determining the unique edge extremal graph for $K_{r+1}$.  In 1962, Erd\H{o}s \cite{Erdos1962}
characterized the unique edge extremal graph for $kK_3$ if $n > 400(k-1)^2$. Later in 1968, Moon \cite{Moon1968} showed that the result of Erd\H{o}s is also valid for $n >\frac{19k-11}{2}$.  Simonovits \cite{Simonovits1966} and Moon \cite{Moon1968}
independently characterized the unique edge extremal graph for $kK_{r+1}$ if  $n$ is sufficiently large.

\begin{theorem}[\cite{Moon1968,Simonovits1966}]\label{thm-Simonovits}
 Let $G$  be a $kK_{r+1}$-free graph of order $n$. If $n$ is  sufficiently large, then $ex(n,kK_{r+1})= e(K_{k-1}\vee  T_r(n-k+1))$, and the unique edge extremal graph is $K_{k-1}\vee  T_r(n-k+1)$.
\end{theorem}

Let $A(G)$ be the adjacency matrix  and
$Q(G)=D(G)+A(G)$  the signless Laplacian matrix of a graph $G$, respectively,
where $D(G)$ is the degree diagonal matrix of $G$. The largest eigenvalue of  $Q(G)$ is called the  {\it signless Laplacian spectral radius} of $G$,  denoted by $q(G)$. 
The signless Laplacian spectral radius has also attracted considerable interest due to its applications in combinatorial optimization, network analysis, and chemical graph theory, particularly in relation to bipartiteness and the presence of odd cycles.

In 2013, de Freitas, Nikiforov and Patuzzi \cite{Freitas-Nikiforov-Patuzzi2013} presented the signless Laplacian spectral Tur\'{a}n-type problem, which asks for the maximum signless Laplacian spectral radius in an $n$-vertex $F$-free graph. 
 The $n$-vertex $F$-free graph with this maximum signless Laplacian spectral radius is called {\it signless Laplacian spectral extremal graph} for $F$.
This problem has attracted much attention in the past decades.
Dating back to \cite{Merris1998},  an upper bound on the
signless Laplacian spectral radius involving the degrees of vertices was provided.
\begin{theorem} [\cite{Feng Yu 2009,Merris1998}]\label{the upper bound of q(G)}
Let $G$ be a graph with vertex set $V(G)$ without isolated vertices. Then
\begin{eqnarray}\label{Merris bound}
q(G)\leq \max\limits_{u\in V(G)} \left\{d_{G}(u)+\frac{1}{d_{G}(u)}\sum\limits_{v\in N_{G}(u)} d_{G}(v)\right\}.
\end{eqnarray}
If $G$ is connected, then equality holds if and only if $G$ is either a regular graph or a semi-regular bipartite graph.
\end{theorem}

The upper bound~(\ref{Merris bound}) is effective for
solving  the signless Laplacian spectral Tur\'{a}n-type problem where the signless Laplacian spectral extremal graph is sparse, particularly those  structurally close to the complete split graph,
for example,
linear forests~\cite{Chen2020},
intersecting odd cycles~\cite{Chen2024},
books~\cite{Chen-Jin-Zhang2025},
complete bipartite graphs~\cite{de Freitas-Nikiforov-Patuzzi 2016},
short odd cycles~\cite{Miao-Liu-Xue2023},
fan graphs~\cite{Wang-Zhai2023},
intersecting triangles~\cite{Zhao-Huang-Guo2021}.
 This bound helps characterize the structure of signless Laplacian spectral extremal graphs
through vertex degrees.
However, for signless Laplacian spectral extremal graphs that are structurally similar to the Tur\'{a}n graphs $T_{r}(n)$, the balanced complete $r$-partite graph of order $n$, this upper bound becomes less effective. It is therefore necessary to develop new methods for investigating
$F$-free graphs.

Denote by  $\pi(F)$  the Tur\'{a}n density of a graph $F,$ which  is defined as $\pi(F)=\lim\limits_{n\rightarrow \infty} \frac{ex(n,F)}{\binom{n}{2}}.$
Recently, Zheng, Li and Li~\cite{Zheng  Li Li 2026}  established a general criterion for
signless Laplacian spectral version of the criterion of Keevash,
Lenz and Mubayi~\cite{Keevash Lenz Mubayi 2014} for $\chi(F)\geq 4$  as follows, where $\chi(F)$ is the chromatic number of $F.$
\begin{theorem}[\cite{Zheng Li Li 2026}]\label{theorem of liyongtao}
Let $r\geqslant 3$ and $F$ be a  graph with Tur\'{a}n density
$\pi(F)=1-\frac{1}{r},$ where $\chi(F)=r+1.$ 
For real numbers $0<\epsilon <\frac{1}{2}$
and $\sigma<\frac{\epsilon}{36}.$
Suppose that
there exists $N > 0$ such that for every $n \geqslant  N,$ we have
\begin{eqnarray*}\label{the first inequality}
\begin{vmatrix}\operatorname{ex}(n,F)-\operatorname{ex}(n-1,F)-\pi(F)n\end{vmatrix}\leqslant \sigma n,
\end{eqnarray*}
and
\begin{eqnarray*}\label{the second inequality}
\begin{vmatrix}q(\mathcal{G}_n)-4\mathrm{ex}(n,F)n^{-1}\end{vmatrix}
\leqslant \sigma.
\end{eqnarray*}
Then there exists $n_{0} \in \mathbb{N}$ such that for any $F$-free graph $H$ on $n \geqslant n_{0}$ vertices, we have
$$
q(H)\leqslant  q(\mathcal{G}_n),
$$
where $\mathcal{G}_{n}$ is the collection of all $n$-vertex $F$-free graphs with
minimum degree more than $(\pi(F)-\epsilon)n$ and $q(\mathcal{G}_{n})=\max\{q(G):G\in \mathcal{G}_{n}\}.$ In addition, if the equality holds, then $H \in \mathcal{G}_n$.
\end{theorem}

As an application, they proved that for any color-critical graph
$F$ with  $\chi(F) = r + 1 \ge 4$, if $n$ is sufficiently large and $G$
is an $F$-free graph on $n$ vertices, then $q(G) \le q(T_{r}(n))$, with
equality if and only if $G = T_{r}(n)$.
In~\cite{Chen-Jin-Zhang-Zhang2025},  employing
the regularity method,  F\"{u}redi's stability theorem and Theorem~\ref{theorem of liyongtao},
the authors proved that for a given graph $F$ and $r\geqslant 3$, if $ex(n, F) = e(T_r(n))+O(1)$, then
signless Laplacian spectral extremal graphs must be edge extremal graphs
for sufficiently large $n$.

Besides,
in 2026, Zheng, Li and Fan~\cite{Zheng  Li Fan 2026} established  the corresponding stability result  for the signless
Laplacian spectral radius with $\chi(F) \geq 4.$
\begin{theorem}[\cite{Zheng  Li Fan 2026}]\label{stability result  for the signless}
Let $F$ be a graph with chromatic
number $\chi(F) = r + 1 \geq 4.$ For every $\epsilon > 0$, there exist $n_0$
and $\sigma > 0$ such that if $G$ is an $F$-free graph on $n \geq n_0$ vertices
with
$q(G) \geq  (1-\frac{1}{r}-\sigma)2n$, then $G$ can be obtained from $T_{r}(n)$ by
adding and deleting at most
$\epsilon n^2$
edges.
\end{theorem}

As  a first application, they firstly
proved that if $r \geq 2,$  $\epsilon>0$ and $G$ is an $n$-vertex graph with $m$ edges
such that $\sum_{v\in V(G)}
d^2(v) \geq 2(1-\frac{1}{r}+\epsilon)mn$, then $G$ contains
a  $T_{r+1}((r+1)t)$ with $t = \Omega _{r,\epsilon} (\log n).$ This generalizes the
Bollob\'{a}s-Erd\H{o}s theorem.
Secondly, they established the spectral extremal results for hypergraphs by showing that for any graph $F$ with $\chi(F) = r + 1$, where $r \geq k \geq 2$, and for any $\varepsilon > 0$, there exists $n_0$ such that if $G$ is a linear $k$-uniform hypergraph on $n \geq n_0$ vertices with $q(G) \geq \left(1 - \frac{1}{r} + \varepsilon\right)\frac{2n}{k - 1}$, then $G$ contains a $k$-expansion of $F$.

Although Theorems~\ref{theorem of liyongtao} and \ref{stability result  for the signless} are powerful new tools for solving signless Laplacian spectral Tur\'{a}n-type problem, they are not applicable for $\chi(F)\leq 3$ and the signless Laplacian spectral extremal graphs  resembling Tur\'{a}n graphs. So it is interesting to consider the signless Laplacian spectral Tur\'{a}n-type problem for $\chi(F)\leq 3$ and the signless Laplacian spectral extremal graph  resembling Tur\'{a}n graphs.




 In 2013, He, Jin and Zhang \cite{He-Jin-Zhang2013} characterized the signless Laplacian spectral extremal graphs for $K_{r+1}$ and $r \geq 2$. In 2024, Zhang and Wang \cite{Zhang-Wang2024} proposed the following problem, which is the  signless Laplacian spectral extremal analogue of  Theorem~\ref{thm-Simonovits}.

\vspace{2mm}
{\noindent\bf Problem.} What is the maximum signless Laplacian spectral radius and what are the corresponding signless Laplacian spectral extremal graphs for $kK_{r+1}$ with $k \geq 2$ and $r \geq 2?$
\vspace{2mm}

Since Theorems~\ref{theorem of liyongtao} and \ref{stability result  for the signless}
fail to cover the case $r=2$, it is worthwhile to investigate the above problem for $r=2$.
In \cite{Zhang-Wang2024}, the authors resolved the problem for $(k,r)=(2,2)$ when $n\geq 44$.
Later, the same authors further settled the case $(k,r)=(3,2)$ for $n\geq 453$ in \cite{Zhang-Wang2024-2}.

In spectral extremal graph theory, partitioning vertices according to large and small eigenvector entries
has become a powerful approach for handling spectral Tur\'{a}n-type problem.
Nevertheless, this vertex partition method is generally  applicable
when
 the edge extremal graphs have $O(n)$ edges
(see, e.g.,~\cite{Chen-Li-Li-Yu-Zhang2023,Chen2024-1,Cioaba Desai Tait 2023,Tait 2019,Tait Tobin 2017}).
By Theorem~\ref{thm-Simonovits},  $\mathrm{ex}(n,kK_{3})=\Theta(n^2)$ for sufficiently large $n$,
which indicates the edge extremal case here is of quadratic order rather than linear order.
In this paper,
we  settle the problem for all $k\geq 2$ and $r=2$ when $n$ is sufficiently large.
Our proof combines the vertex partition technique based on large and small eigenvector entries
with the upper bound~\eqref{Merris bound}.
Our result can be stated as follows.

\begin{theorem}\label{main result}
Let $k\geq 2$ and  $G$ be a $kK_{3}$-free graph of sufficiently large order $n.$
Then
$$q(G)\leq q\left(K_{k-1}\vee T_2(n-k+1)\right),$$
and equality holds if and only if $G=K_{k-1}\vee T_2(n-k+1).$
\end{theorem}

The organization of this paper is outlined as follows. In Section~2, some notations and useful lemmas are provided.   In Section~3, we give the proof of Theorem~\ref{main result}.  In Section~4, a generalized conjecture is proposed.

\section{Preliminaries}
In this section, we will list some symbols and useful lemmas.
Let $G=(V(G),E(G))$ be a simple graph.
We call $|V(G)|$ the {\it order} of $G $ and $|E(G)|$ the {\it size} of $G,$
denoted by $e(G).$
For two vertices $u,v\in V(G),$ $u$ and $v$ are {\it adjacent} if $uv\in E(G).$
The set of neighbors of $v$ in $G$ will be written as $N_{G}(v).$
Denote by $d_{G}(v)$  the {\it degree} of vertex $v\in V (G)$ and $\Delta(G)$  the {\it maximum degree} of vertices in $G.$
For two disjoint vertex sets $A,B\subseteq V(G),$  $e(A)$ denotes the number of edges of $G$ whose endpoints are both in $A,$ and
$e(A,B)$ denotes the number of edges of $G$ with one end vertex in $A$ and the other  in $B.$
The {\it induced subgraph} $G[A]$   is a subgraph of $G$ with vertex set $A$ and edge set consisting of all edges with both endpoints in $A.$
For $S\subseteq V(G),$ denote by $G-S$ the graph obtained from $G$ by deleting all vertices in $S$ and  all edges incident to the vertices in $S$. Specially, when $S=\{v\},$ we use $G-v$ instead of $G-\{v\}.$

For two vertex disjoint graphs $G$ and $H,$ denote by $G\cup H$ the {\it union} of $G$ and $H,$
denote
the {\it join} of $G$ and $H$
by $G\vee H,$ which is obtained from $G$ and $H$ by adding all edges between  $V(G)$ and $V(H).$ Especially, write $kG$ for $k$ disjoint union of $G$ for brevity.
Denote by $\overline{G}$ the {\it complement} of $G.$ Readers are referred to the book~\cite{Bondy Murty2008} for undefined notations and definitions.

The independence number of $G$, denoted $\alpha(G)$, is the order of the largest independent
set in $G$.
For positive integers $k$ and $\ell$, the Ramsey number $R(k,\ell)$ is the smallest integer $n$
such that if we arbitrarily color the edges of the complete graph $K_{n}$ with $2$ colors, then
it contains a monochromatic $K_{k}$ in the first color or a monochromatic $K_{\ell}$
in the second
color. If the edges in the first color are interpreted as a graph $G$ and those in the second
color as its complement $\overline{G}$, then $R(k, \ell)$ can be defined equivalently as the smallest $n$ such
that every graph on $n$ vertices contains $K_{k}$ or has independence number $\alpha(G) \geq  \ell$.
\begin{lemma}[\cite{Greenwood  Gleason 1955}]\label{Ramsey theorem}
$R(k,\ell)\leq \binom{k+\ell-2}{k-1}.$
\end{lemma}

\begin{lemma}[\cite{Erdos Gallai 1959,Moon1968}]\label{dege number of disjoint matchings}
Let $k \ge 2$ and $G$ be a $kK_2$-free graph of order $n\geq 2k-1$. Then
\[
e(G) \le
\begin{cases}
(k-1)n - \dfrac{k(k-1)}{2}, & \text{for } n \ge \dfrac{5k-2}{2}; \\[1em]
\dbinom{2k-1}{2}, & \text{for } 2k-1 \le n < \dfrac{5k-2}{2}.
\end{cases}
\]
If $n > \dfrac{5k-2}{2},$ then equality holds if and only if $G = K_{k-1} \vee \overline{K}_{n-k+1}$; if $n = \dfrac{5k-2}{2},$ then equality holds if and only if $G \in \{K_{2k-1} \cup \overline{K}_{n-2k+1}, K_{k-1} \vee \overline{K}_{n-k+1}\}$; if $2k-1 \le n < \dfrac{5k-2}{2},$ then equality holds if and only if $G = K_{2k-1} \cup \overline{K}_{n-2k+1}$.
\end{lemma}

\begin{lemma}[\cite{Cvetkovic Rowlinson Simic 2007}]\label{subgraph}
Let $G$ be a connected graph. If $H$ is a proper subgraph of $G$, then
\[
q(H) < q(G).
\]
\end{lemma}

The {\it spectral radius} of a  nonnegative real matrix $A,$
denoted by $\lambda(A),$ is  the largest  eigenvalue of $A.$
Since $Q(G)$ is a
nonnegative real matrix, by the Perron-Frobenius theorem, there is a
nonnegative eigenvector corresponding to  $q(G)$ with maximum component $1.$ We call this nonnegative eigenvector 
the {\it Perron vector} of $G.$

\begin{lemma}[Lemma~2.3.1,\cite{Brouwer2012}]\label{equitable-1}
Let $A$ be an $n\times n$ symmetric real matrix. 
Assume that $B$ is the  quotient matrix of $A$ with respect to an equitable partition $\pi=\{X_1, \ldots, X_m\}$ and $S$ is the $n\times m$ matrix whose $j$-th column is the characteristic vector of $X_{j}$ $(j=1,\ldots,m).$ If $\mathbf{v}$ is an eigenvector of $B$ for an eigenvalue $\lambda,$ then $S\mathbf{v}$ is an eigenvector of $A$ for the same eigenvalue $\lambda.$
\end{lemma}

\begin{lemma}[Theorem~8.3.4,\cite{Horn Johnson 2013}]\label{equitable-2}
Let $A$ be a nonnegative real matrix. If there is a positive eigenvector $\mathbf{x}$ and a nonnegative real number $\lambda$ such that either $A\mathbf{x}=\lambda \mathbf{x}$ or $\mathbf{x}^{\mathrm{T}}A=\lambda \mathbf{x}^{\mathrm{T}},$ then $\lambda=\rho(A).$
\end{lemma}


\begin{lemma}[\cite{Zheng  Li Fan 2026}]\label{upper bound of q}
 If $k \ge 2$, $\varepsilon > 0$ and $G$ is an $n$-vertex graph with
\[
q(G) \ge \left(1 - \frac{1}{k} + \varepsilon\right) 2n,
\]
then $G$ contains a copy of $T_{k+1}((k+1)t)$ with $t = \Omega_{k,\varepsilon}(\log n)$.
\end{lemma}

\begin{lemma}[\cite{Zhang-Wang2024-2}]\label{join}
Let $t \geq 1$, $k \geq t + 1$ and $G$ be a graph of order $n \geq 3k$. Suppose that $G = K_t \vee H$, where $H$ is a $(k-t)K_{3}$-free graph of order $n-t$. Then $G$ is a $kK_{3}$-free graph.
\end{lemma}

\begin{lemma}[\cite{Cioaba Feng Tait Zhang 2020}]\label{property of cup of sets}
Let $V_{1},\ldots, V_{p}$ be finite sets. Then $$ |V_{1}\cap\cdots\cap V_{p}|\geqslant \sum_{i=1}^{p}|V_{i}|-(p-1)\mid\bigcup_{i=1}^{p}V_{i}\mid.$$
\end{lemma}

\begin{lemma}\label{lower bound}
Let $n$ be sufficiently large and $G=K_{k-1}\vee T_2(n-k+1).$ Then
 $$q(G)>n-2+\sqrt{(k-1)n}.$$
\end{lemma}
\begin{proof}
We divide the proof into the following two cases.

{\bf Case~1.} $n-k$ is odd. Then $G=K_{k-1}\vee K_{\frac{n-(k-1)}{2}, \frac{n-(k-1)}{2}}$
and  $Q(G)$ has the equitable quotient matrix
$$B_{\Pi}=\left[\begin{array}{cc}
n+k-3 & n-k+1 \\
k-1 & n
\end{array}\right].$$
Therefore, the characteristic polynomial of $B_{\Pi}$ is
$$f(x)=x^2-(2n+k-3)x+(n+k-3)n-(k-1)(n-k+1).$$
Then the largest root of $f(x)=0$ is $\frac{2n+k-3+\sqrt{(2n+k-3)^2-4((n-1)^2+k^2-2k)}}{2}>n-2+\sqrt{(k-1)n}$ for sufficiently large $n.$
By the equitable partition, $q(G)=\lambda(B_{\Pi})$ is the largest root of $f(x)=0.$

{\bf Case~2.} $n-k$ is even. Then
$G=K_{k-1}\vee K_{\frac{n-(k-1)-1}{2}, \frac{n-(k-1)+1}{2}}$
and $Q(G)$ has the equitable quotient matrix
$$B_{\Pi'}=\left[\begin{array}{ccc}
n+k-3 & \frac{n-k}{2} &  \frac{n-k+2}{2} \\
k-1 & \frac{n+k}{2} & \frac{n-k+2}{2}  \\
k-1   & \frac{n-k}{2}     &  \frac{n+k-2}{2} \\
\end{array}\right].$$
By a simple calculation,  the characteristic polynomial of $B_{\Pi'}$ is
$$g(x)=x^3-2(n+k-2)x^2+(n^2+2(k-2)n+2(k^2-3k+2))x-(k-1)n^2+2(k-1)n-k^3+3k^2-2k.
$$
By the equitable partition, $q(G)=\lambda(B_{\Pi'})$ is the largest root of $g(x)=0.$

Besides,
we have
$$g(n-2+\sqrt{(k-1)n})=(k+1)(2k-n)\sqrt{(k-1)n}-k(k-1)(k+2)< 0$$
for sufficiently large $n.$
Thus, we have $q(G)>n-2+\sqrt{(k-1)n}.$
\end{proof}

\begin{lemma}\label{maximum component implies maximum degree}
Let $\mathbf{x}$ be the Perron vector of $G.$ If $w^{*}\in V(G)$ is a vertex with $x_{w^{*}}=1$ and  $u^{*}\in V(G)$ is a vertex  with $d_{G}(u^{*})=n-1,$
then $d_{G}(w^{*})=n-1$ and $x_{u^{*}}=1.$
\end{lemma}
\begin{proof}
If $w^{*}=u^{*},$ then the conclusion is immediate. Hence we assume $w^{*}\neq u^{*}$.
Since $\Delta(G)=n-1,$ $G$ contains a copy of $K_{1,n-1}.$ By Lemma~\ref{subgraph}, we have $q(G)\geq n.$
By the eigenvalue equation for $Q(G),$ we have
$$q(G)x_{w^{*}}=d_{G}(w^{*})x_{w^{*}} +\sum_{v\in N_{G} (w^{*})}x_{v}=d_{G}(w^{*})x_{w^{*}} +\sum_{\substack{v\in N_{G}( w^{*})\\v\neq u^{*}}}x_{v}+x_{u^{*}},$$
which implies that
\begin{eqnarray}\label{first equality}
\sum_{\substack{v\in N_{G}( w^{*})\\v\neq u^{*}}}x_{v}=(q(G)-d_{G}(w^{*}))x_{w^{*}}-x_{u^{*}}.
\end{eqnarray}
By (\ref{first equality}) and $d_{G}(u^{*})=n-1,$  we have
\begin{eqnarray*}
q(G)x_{u^{*}}&=&(n-1)x_{u^{*}} +\sum_{v\in V(G)\setminus \{u^{*}\}}x_{v} 
=(n-1)x_{u^{*}}+\sum_{\substack{v\in N_{G}(w^{*})\\ v\neq u^{*}}}x_{v}+x_{w^{*}}+\sum_{\substack{v\notin N_{G}(w^{*})\\ v\neq w^{*}}}x_{v}\nonumber \\
&=& (n-1)x_{u^{*}}+ (q(G)-d_{G}(w^{*}))x_{w^{*}}-x_{u^{*}}+x_{w^{*}}+\sum_{\substack{v\notin N_{G}(w^{*})\\ v\neq w^{*}}}x_{v}\nonumber \\
&=& (n-2)x_{u^{*}}+ (q(G)-d_{G}(w^{*})+1)x_{w^{*}}+\sum_{\substack{v\notin N_{G}(w^{*})\\ v\neq w^{*}}}x_{v},
\end{eqnarray*}
which implies that
$$(q(G)-n+2)x_{u^*}=(q(G)-d_{G}(w^*)+1)x_{w^*}+\sum_{\substack{v\notin N_{G}(w^{*})\\ v\neq w^{*}}}x_{v}.$$
Thus, by $q(G)\geq n$ and $x_{w^*}\geq x_{u^*},$ we have
$$x_{w^*}\geq x_{u^*}\geq \frac{q(G)-n+2}{q(G)-d_{G}(w^*)+1}x_{u^*}=x_{w^*}+\frac{\sum_{\substack{v\notin N_{G}(w^{*})\\ v\neq w^{*}}}x_{v}}{q(G)-d_{G}(w^*)+1}\geq x_{w^*},$$
which implies that $d_{G}(w^*)= n-1$ and  $x_{u^*}=1.$
We complete the proof.
\end{proof}

\begin{lemma}\label{larger component derive larger degree}
Let $G$ be a graph of order $n$ with maximum degree $n-1$ and
 $\mathbf{x}$ be the Perron vector of $G.$
 If $x_{u}>\frac{3}{4}$ for $u\in V(G),$ then $d_{G}(u)>\frac{4(n-2)}{3}-\frac{q(G)}{3}.$
\end{lemma}
\begin{proof}
Since $\Delta(G)=n-1,$ $G$ contains a copy of $K_{1,n-1}.$ By Lemma~\ref{subgraph}, we have $q(G)\geq n.$
Let
$w^{*}\in V(G)$ be a vertex with $x_{w^{*}}=1.$
By Lemma~\ref{maximum component implies maximum degree}, $d_{G}(w^{*})=n-1.$ Thus, for any $u\in V(G)$ with $x_{u}>\frac{3}{4},$ we have
\begin{eqnarray*}
q(G)x_{w^{*}}&=&(n-1)x_{w^{*}}+\sum_{v\in V(G)\setminus \{w^{*}\}}x_{v} 
= (n-1)x_{w^{*}}+\sum_{v\in N_{G}( u)}x_{v} +\sum_{v\notin N_{G}( u) }x_{v}- x_{w^{*}} \nonumber \\
&=& (n-2)x_{w^{*}}+(q(G)-d_{G}(u))x_{u} +\sum_{v\notin N_{G}( u)}x_{v} ,
\end{eqnarray*}
which implies that
\begin{eqnarray*}
d_{G}(u)&=&q(G)-\frac{q(G)-(n-2)}{x_{u}}+\frac{\sum_{v\notin N_{G}( u)}x_{v}}{x_{u}}
>q(G)-\frac{4(q(G)-(n-2))}{3}+\frac{\sum_{v\notin N_{G}( u)}x_{v}}{x_{u}}\\
&\geq&\frac{4(n-2)}{3}-\frac{q(G)}{3},
\end{eqnarray*}
where the first inequality holds as $q(G)\geq n$ and $x_{u}>\frac{3}{4}.$
We complete the proof. 
\end{proof}

For any two non-adjacent vertices $u$ and $v$ in $G$, if we delete  all the incident edges of $u$ and add all new edges connecting $u$ to $N_G(v)$, then such an operation is called {\it symmetrization} and we say ``{\it symmetrize $u$ to $v$}" for simplicity.

\begin{lemma}\label{symmetrize and adding edges}
Let $G$ be a  connected graph whose vertex set is partitioned into four 
subsets $A,B,C$ and $D$, and let $\mathbf{x}$ be a positive eigenvector corresponding to $q(G)$.
Let $x_u = \max\{x_v : v \in A\cup C\},$   $B\subseteq N_G(u)$, $C_0 = C \cap N_G(u)$ and $C_1 = C \setminus C_0$.
If  $A$ 
is an independent set and $u \in A$,
then
\[
q(G) \leq q\left(K_{|D|} \vee \overline{K}_{|A|+|C_1|}\vee  G[B\cup C_0]\right),
\]
and equality holds if and only if $G=K_{|D|} \vee \overline{K}_{|A|+|C_1|}\vee  G[B\cup C_0]$.
\end{lemma}

\begin{proof}
Construct a new graph $G'$ from $G$ as follows.
\begin{itemize}
    \item For every vertex $v \in (A \cup C_1)\setminus \{u\}$, symmetrize $v$ to  $u$.
    \item For every vertex $v \in D$, add all possible edges between $v$ and the remaining vertices of $V(G)$.
\end{itemize}
Since $A$ is an independent set and $B\subseteq N_G(u)$,
$
G' = K_{|D|} \vee \overline{K}_{|A|+|C_1|}\vee  G[B\cup C_0]
$.

Define a new vector $\mathbf{y}$ on $V(G)$ by
\[
y_v =
\begin{cases}
x_u, & v \in A \cup C_1,\\[4pt]
x_v, & v \notin A \cup C_1.
\end{cases}
\]

 In the sequel, we verify the inequality $q(G) y_v \le (Q(G')\mathbf{y})_v$ for every vertex $v\in V(G)$.

{\bf Case 1.} $v \in A \cup C_1$.
By the definitions of $\mathbf{x}$ and $\mathbf{y}$,
\[
q(G) y_v = q(G) x_u = (Q(G)\mathbf{x})_u = d_G(u) x_u + \sum_{w \in N_G(u)} x_w.
\]
By the construction of $G'$, $d_{G'}(u)=d_{G'}(v)$ and $N_{G'}(u)=N_{G}(u)\cup D=B\cup C_0\cup D$.
Then  $d_G(u)\leq d_{G'}(u)$ and  \[
q(G)y_v=q(G) x_u \le d_{G'}(u) x_u + \sum_{w \in N_{G'}(u)} x_w
= d_{G'}(v) y_v + \sum_{w \in N_{G'}(v)} y_w
= (Q(G')\mathbf{y})_v,
\]
and equality holds if and only if $N_G(u)=N_{G'}(u) =B\cup C_0 \cup D$.

{\bf Case 2.} $v \in D$.
By the definitions of $\mathbf{x}$ and $\mathbf{y}$,
\[
q(G) y_v = q(G) x_v = (Q(G)\mathbf{x})_v = d_G(v) x_v + \sum_{w \in N_G(v)} x_w.
\]
Since  $x_v=y_v$ and $v$ is adjacent to all other vertices in $G'$,
\[
q(G)y_v=q(G) x_v
\le (n-1) y_v + \sum_{w \in N_{G'}(v)} y_w
= (Q(G')\mathbf{y})_v,
\]
and  equality holds  if and only if $d_G(v)=n-1$
 and $x_u=x_w$  for any $w\in A\cup C_1$.

{\bf Case 3.} $v \in B \cup C_0$.
By the definitions of $\mathbf{x}$ and $\mathbf{y}$,
\[
q(G) y_v = q(G) x_v = (Q(G)\mathbf{x})_v = d_G(v) x_v + \sum_{w \in N_G(v)} x_w.
\]
By the construction of $G'$, $ N_{G'}(v)=N_{G}(v)\cup A\cup C_1\cup D$. 
Together with $x_v=y_v$ and $x_u \ge x_w$ for any $w \in A \cup C_1$,
\begin{eqnarray*}
q(G) y_v=q(G) x_v
\le  d_{G'}(v) y_v + \sum_{w \in N_{G'}(v)} y_w
= (Q(G')\mathbf{y})_v,
\end{eqnarray*}
and the  equality holds if and only if $d_G(v)=d_{G'} (v)$, i.e., $ N_G(v)=N_{G'}(v)=N_{G}(v)\cup A\cup C_1\cup D$, and $x_u=x_w$  for any $w\in A\cup C_1$.  Clearly, the  equality holding implies that $v$ is adjacent to all vertices in $A\cup C_1\cup D$.

Hence,  $q(G)\mathbf{y} \le Q(G')\mathbf{y}$ for the positive vector $\mathbf{y}$.
By  the Perron-Frobenius theorem, $
q(G') \ge q(G)$ and if the equality holds
 then $q(G)\mathbf{y}=Q(G')\mathbf{y}$, which implies that all equalities hold in Cases 1-3. Hence $G=K_{|D|} \vee \overline{K}_{|A|+|C_1|}\vee  G[B\cup C_0]$.  Conversely,  if $G=K_{|D|} \vee \overline{K}_{|A|+|C_1|}\vee  G[B\cup C_0]$, then obviously $q(G)=q(K_{|D|} \vee \overline{K}_{|A|+|C_1|}\vee  G[B\cup C_0])$.
\end{proof}

Denote by $K_{n_{1},\ldots,n_{s}}$  the {\it complete $s$-partite} graph of partition sizes $n_{1},\ldots,n_{s},$  where $n_{1}\geq \ldots\geq n_{s}\geq 0.$
\begin{lemma}\label{JinLe-2}
Let $G_1=K_{n_1,\ldots,n_s}\vee H$ and $G_2=K_{n_1-1,\ldots,n_s+1}\vee H$, where $H$ is a graph of order $t\geq 1$, $n_1\ge  \cdots\ge n_s\geq 0$, $n_{s-1}\geq 1,$   $n_1-n_s\ge 2$ and $s\ge 2$. Then $q(G_2)> q(G_1)$.
\end{lemma}
\begin{proof}
If $n_{s}= 0,$ then $G_{1}$ is a proper subgraph of $G_{2},$ which implies that $q(G_2)> q(G_1).$   In the sequel, we assume that $n_{s}\neq 0.$
Let $n=t+\sum_{i=1}^{s}n_{i}$ and $\pi_1$(resp.~$\pi_2$) be the partition of $V(G_1)$(resp.~$V(G_2)$)  consisting of partite sets of 
$K_{n_1,\ldots,n_s}$
(resp.~$K_{n_1-1,\ldots,n_s+1}$) and each vertex of $H$. Then $\pi_1$ and $\pi_2$ are equitable partitions corresponding to matrices $Q(G_1)$ and $Q(G_2),$ respectively, and
$$Q(G_1)/{\pi_1}=\left(\begin{matrix}
    Q_{11}&J_{s\times t}\\
    Q_{21}&Q(H)+\left(\sum_{i=1}^{s}n_{i}\right)I
\end{matrix}\right),~~Q(G_2)/{\pi_2}=\left(\begin{matrix}
    Q_{11}'&J_{s\times t}\\
    Q_{21}'&Q(H)+\left(\sum_{i=1}^{s}n_{i}\right)I
\end{matrix}\right),$$
where $I$ is the identity matrix of order $t,$
$$Q_{11}=\left(\begin{matrix}
    n-n_1&n_2&\cdots& n_s\\
    n_1&n-n_2&\cdots& n_s\\
    \vdots&\vdots& &\vdots\\
    n_1&n_2&\cdots&n-n_s
\end{matrix}\right),~~Q_{11}'
=\left(\begin{matrix}
    n-(n_1-1)&n_2&\cdots& n_s+1\\
    n_1-1&n-n_2&\cdots& n_s+1\\
    \vdots&\vdots& &\vdots\\
    n_1-1&n_2&\cdots&n-(n_s+1)
\end{matrix}\right),$$
and
$$Q_{21}=\left(\begin{matrix}
    n_1&n_2&\cdots& n_s\\
    n_1&n_2&\cdots& n_s\\
    \vdots&\vdots& &\vdots\\
    n_1&n_2&\cdots&n_s
\end{matrix}\right),~~Q_{21}'
=\left(\begin{matrix}
    n_1-1&n_2&\cdots& n_s+1\\
    n_1-1&n_2&\cdots& n_s+1\\
    \vdots&\vdots& &\vdots\\
    n_1-1&n_2&\cdots&n_s+1
\end{matrix}\right).$$

It follows from the irreducibility of $Q(G_1)$ and $Q(G_2),$
Lemmas~\ref{equitable-1} and \ref{equitable-2} that
$q(G_i)=\lambda(Q(G_i)/\pi_{i})$ for $i=1,2.$
By the Perron-Frobenius theorem, there exist  two positive vectors $\mathbf{x},\mathbf{y}$ satisfying  $\mathbf{x}^{\mathrm{T}}Q(G_1)/\pi_1=q(G_1)\mathbf{x}^{\mathrm{T}},$ $Q(G_2)/\pi_2 \mathbf{y}=q(G_2)\mathbf{y}$ and $\sum_{i=1}^{s+t}x_i=1$. 
Hence,
\begin{eqnarray}\label{Jin-1}
    (q(G_1)-q(G_2))\mathbf{x}^{\mathrm{T}}\mathbf{y}&=&\mathbf{x}^{\mathrm{T}}[Q(G_1)/\pi_1-Q(G_2)/{\pi_2}]\mathbf{y}\nonumber\\
    &=&\left(\sum_{i=2}^{s+t}x_i- x_1\right)y_1+\left(x_s-\left(\sum_{i=1}^{s-1}x_i+\sum_{i=s+1}^{s+t}x_i\right)\right)y_s \nonumber\\
     &=& (1- 2x_1)y_1+(2x_s-1)y_s.
\end{eqnarray}
By characteristic  equations $\mathbf{x}^{\mathrm{T}}Q(G_1)/\pi_1=q(G_1)\mathbf{x}^{\mathrm{T}}$ and $Q(G_2)/\pi_2 \mathbf{y}=q(G_2)\mathbf{y}$,
$$\begin{matrix}
q(G_1)x_1 =(n-2n_1)x_1+n_1, &q(G_1)x_s  =(n-2n_s)x_s+n_s,
\end{matrix}$$
$$q(G_2)y_{1} =(n-(n_1-1))y_1+\sum\limits_{i=2}^{s-1} n_{i}y_{i}  +(n_{s}+1)y_{s}  +\sum\limits_{i=s+1}^{s+t}y_{i},$$
$$q(G_2)y_{s}  =(n_{1}-1)y_{1}+\sum\limits_{i=2}^{s-1}n_{i}y_{i}  +(n-(n_{s}+1))y_{s}+\sum\limits_{i=s+1}^{s+t}y_{i},$$
which implies that
$$x_{1}=\frac{n_{1}}{q(G_{1})-(n-2n_{1})}, ~~ x_{s}=\frac{n_{s}}{q(G_{1})-(n-2n_{s})}, $$
$$q(G_2)(y_{1}-y_{s})=(n-2(n_1-1))y_1-(n-2(n_s+1))y_s.$$
Furthermore,
$$ y_1=\frac{q(G_2)-n+2(n_s+1)}{q(G_2)-n+2(n_1-1)}y_s.$$
Thus $x_1<\frac{1}{2},$ $x_1>x_s$ and $y_1\le y_s$ as  $n_{1}>n_s> 0,$ $q(G_1)> n,$ $q(G_2)>n$ and $q(G_2)-n+2(n_1-1)\geq q(G_2)-n+2(n_s+1) $. Combining with the equation $(\ref{Jin-1})$,
\begin{eqnarray*} (q(G_1)-q(G_2))\mathbf{x}^{\mathrm{T}}\mathbf{y}&=&(1- 2x_1)y_1+(2x_s-1)y_s\leq (1- 2x_1)y_s+(2x_s-1)y_s\\
&=&2(x_s-x_1)y_s  <0,
\end{eqnarray*}
which implies that $ q(G_1)<q(G_2).$
\end{proof}

\begin{lemma}\label{the upper bound of q} Let  $c \geq 1$ be a constant independent of $n$ and $1\leq a\leq  \frac{n}{16c}$ be  two integers.    If   $n$ is  sufficiently large, then
 $q(K_{c}\vee K_{a,n-a-c})<n+\frac{\sqrt{n}}{2}.$
\end{lemma}
\begin{proof}
Let $G=K_{c}\vee K_{a,n-a-c}$ and $\pi$ be the partition of $V(G)$ consisting of the vertex set of $K_{c}$ and two partite sets of
$K_{a,n-a-c}.$ Then $\pi$ is an  equitable partition corresponding to the matrix $Q(G).$ So
 $Q(G)$ has the equitable quotient matrix
$$B_{\Pi}=\left[\begin{array}{ccc}
n+c-2 & a &  n-a-c \\
c & n-a & n-a-c  \\
c   & a     &  a+c \\
\end{array}\right].$$
By a simple calculation,  the characteristic polynomial of $B_{\Pi}$ is
$$h(x)=x^3-2(n+c-1)x^2+(2nc+n^2+2c^2-2n-2c)x+4ac^2+2nc-2nc^2+4ca^2-4cna.$$
By Lemmas~\ref{equitable-1} and \ref{equitable-2}, $q(G)=\lambda(B_{\Pi})$ is the largest root of $h(x)=0.$
Furthermore, the largest root of
$h'(x)=3x^2-4(n+c-1)x+n^2+2(nc+c^2-n-c)=0$ is $x_0=\frac{2(n+c-1)+\sqrt{(n+c-1)^2-3(c^2-1)}}{3}\leq n+c-1<n+\frac{\sqrt{n}}{2}$ as $n$ is sufficiently large. Thus $h(x)$ is monotone increasing in  $[x_0, +\infty)$.

Set $\alpha=n+\sqrt{n}/2$. A direct calculation gives
\begin{align*}
h(\alpha)
&=\frac{n^2}{4}
+\left(\frac{9}{8}-c\right)n^{3/2}
+\frac{1-c}{2}n+c(c-1)\sqrt{n}
+4ac(a+c-n).
\end{align*}
Regarding the right-hand side as a function of $a$,
its derivative is
\[
4c(2a+c-n)<0
\]
for $1\le a\le n/(16c)$ and sufficiently large $n$.
Consequently,
\begin{align*}
h(\alpha)
&\ge \frac{n^2}{64c}
+\left(\frac{9}{8}-c\right)n^{3/2}
+\left(\frac{1}{2}-\frac{c}{4}\right)n
+c(c-1)\sqrt{n}>0,
\end{align*}
where the last inequality holds because $c$ is fixed
and $n$ is sufficiently large.
Since $x_0<\alpha$ and $h(x)$ is strictly increasing
on $[\alpha,\infty)$, its largest root is less
than $\alpha$. Hence $
q(G)<n+\frac{\sqrt{n}}{2}.
$
\end{proof}

\section{Proof of Theorem~\ref{main result}}

Throughout this section, $k\geq 2$ is fixed and $n$ is sufficiently large in terms of $k.$ Choose  $\epsilon=\frac{1}{1000k^2}$ and $\epsilon'=\frac{1}{96(k-1)}.$
Let $G$ be a signless Laplacian spectral extremal graph for  $kK_3$ of  order $n$. Set $q = q(G)$ for short.
Let $\mathbf{x}$ be the Perron vector of $G$
and  $u^{*}\in V(G)$ be a vertex with $x_{u^{*}}=1.$ 
Next we prove the following structural lemmas for $G$.

\begin{lemma}\label{}
$G$ is connected.
\end{lemma}
\begin{proof}
Assume for a contradiction that $G$ is not connected.
Assume that $G_{1},\cdots, G_{s}$ are $s$ connected components of $G$ with $q(G_{1})=q.$
Let $H$ be the graph obtained from $G_{1}\cup G_{2}$ by adding an edge between $G_{1}$  and $G_{2}$. Clearly, $H\cup G_{3}\cup \cdots  \cup G_{s}$ is still $kK_3$-free. However, by Lemma~\ref{subgraph}, we obtain $q(H\cup G_{3}\cup \cdots  \cup G_{s})=q(H) > q(G_{1})=q$, a contradiction.
\end{proof}

\begin{lemma}\label{maximum degree is n-1}
$\Delta(G)=n-1.$
\end{lemma}
\begin{proof}
Let $z \in V(G)$ such that
\begin{eqnarray}\label{equality}
d_G(z) + \frac{1}{d_G(z)} \sum_{v \in N_G(z)} d_G(v) = \max\left\{ d_G(u) + \frac{1}{d_G(u)} \sum_{v \in N_G(u)} d_G(v) : u \in V(G) \right\}.
\end{eqnarray}
We first prove the following claims.
\begin{claim}\label{independent edges}
$G[N_G(z)]$ contains at least $\left\lfloor\frac{\sqrt{(k-1)n}}{2}\right\rfloor$ independent edges.
\end{claim}
\begin{proofclaim}~{\bf \ref{independent edges}.}
Let $t=\sqrt{(k-1)n}-2$ and $m=\left\lfloor\frac{t+2}{2}\right\rfloor.$
Since $K_{k-1} \vee T_{2}(n-k+1)$ is a $kK_3$-free graph, by Lemma~\ref{lower bound}, we have
\begin{eqnarray}\label{lower bound of q}
q \ge q\left(K_{k-1} \vee T_{2}(n-k+1)\right) > n +\sqrt{(k-1)n}-2=n+t.
\end{eqnarray}

We first assert that $d_G(z) \geq 2m$. Since $G$ is connected and $\Delta(G) \le n-1$, by (\ref{lower bound of q}) and Theorem~\ref{the upper bound of q(G)}, we obtain
\[
n +t < q \le d_G(z) + \frac{1}{d_G(z)} \sum_{v \in N_G(z)} d_G(v) \le d_G(z) + \Delta(G) \le d_G(z) + n - 1,
\]
which deduces that $d_G(z) > t+1\geq 2m-1.$ So $d_G(z) \geq 2m$.

Now we prove Claim~\ref{independent edges} by way of contradiction. Assume that $G[N_{G}(z)]$ is $m K_{2}$-free. Now we divide it into two cases.

{\bf Case~1}. $2m\le d_G(z) < \frac{5m-2}{2}.$
By Lemma~\ref{dege number of disjoint matchings}, we have
$e(N_{G}(z))\leq \binom{2m-1}{2}.
$
Combining this with (\ref{lower bound of q}) and Theorem~\ref{the upper bound of q(G)}, we obtain
\begin{eqnarray*}
n + t < q &\le& d_G(z) + \frac{1}{d_G(z)} \sum_{v \in N_G(z)} d_G(v) \\
&=& d_G(z) + 1 + \frac{2e(N_G(z)) + e(N_G(z), V(G) \setminus (N_G(z) \cup \{z\}))}{d_G(z)} \\
&\le& d_G(z) + 1 + \frac{2e(N_G(z)) + d_G(z)(n - 1 - d_G(z))}{d_G(z)} \\
&=& n + \frac{2e(N_G(z))}{d_G(z)}
\le  n + \frac{(2m-1)(2m-2)}{d_G(z)}<n+\frac{(t+1)t}{t+1}= n +t ,
\end{eqnarray*}
a contradiction.

{\bf Case~2}. $ d_G(z)\geq \frac{5m-2}{2}.$
By Lemma~\ref{dege number of disjoint matchings}, we have
$
e(N_{G}(z))\leq  \left(m-1\right)d_G(z)-\frac{m(m-1)}{2}.
$
Combining this with (\ref{lower bound of q}) and Theorem~\ref{the upper bound of q(G)}, we obtain
\begin{eqnarray*}
n +t< q &\le& d_G(z) + \frac{1}{d_G(z)} \sum_{v \in N_G(z)} d_G(v) \\
&\leq&
n + \frac{2e(N_G(z))}{d_G(z)}
\leq  n +2(m-1)-\frac{m(m-1)}{d_{G}(z)}
<n+t,
\end{eqnarray*}
a contradiction.
\end{proofclaim}

\begin{claim}\label{property of neighbors}
$G-z$ is $(k-1)K_3$-free.
\end{claim}
\begin{proofclaim}~{\bf \ref{property of neighbors}.}
Suppose to the contrary that $G-z$ contains a copy of $(k-1)K_3,$ denoted by $H.$ Note that the  vertices of $H$ are incident to at most $3(k-1)$ independent edges in $G[N_G(z)]$. By Claim 3.1, there are at least $3k-2$ independent edges in $G[N_G(z)]$ as  $n$ is sufficiently large. Hence, there exists an edge $uv \in E(G[N_G(z)])$ such that $u, v \in N_G(z) \setminus V(H)$. Thus, $G[\{z, u, v\}]\cup H$ is   a copy of $kK_{3}$ in $G$,   a contradiction.
\end{proofclaim}

\begin{claim}\label{degree of n-1}
$d_G(z) = n - 1$.
\end{claim}
\begin{proofclaim}~{\bf \ref{degree of n-1}.}
Suppose to the contrary that $d_G(z) \le n - 2$. Let
\[
G' = G + \{zv : v \in V(G)\setminus (N_G(z)\cup \{z\})\}.
\]
Obviously, $G' = K_1 \vee (G-z)$, where $K_1$ corresponds to the vertex $z$. According to Claim~\ref{property of neighbors}, $G-z$ is $(k-1)K_3$-free, and so $G'$ is $kK_3$-free by Lemma~\ref{join}. However, by Lemma~\ref{subgraph}, we obtain $q(G') > q$, a contradiction.
\end{proofclaim}

By Claim~\ref{degree of n-1}, we complete the proof of Lemma~\ref{maximum degree is n-1}.
\end{proof}

Let $A=\{v\in V(G)\  | \ x_{v}>\frac{3}{4}\},$ $B=\bigcap_{v\in A}N_{G}(v)$ and $C=V(G)\setminus (A\cup B).$

Adding any edge $e\notin E(H)$ to any connected graph $H$ will increase the signless Laplacian spectral radius of $H$ and the  maximum number of disjoint  triangles increases at most one.  Since $G$ is extremal and connected, it contains a copy of $(k-1)K_{3}.$

Since $G$ is $kK_{3}$-free and $K_{k-1}\vee T_{2}(n-k+1)$ is  $kK_{3}$-free, by Lemmas~\ref{upper bound of q} and~\ref{lower bound}, for  sufficiently large $n$, we have
\begin{eqnarray}\label{upper and lower bounds of q}
n+\sqrt{(k-1)n}-2<q<n+\epsilon n.
\end{eqnarray}

\begin{lemma}\label{the upper bound of size A}
$|A|\leq k-1.$
\end{lemma}
\begin{proof}
We prove it by way of contradiction. Let $|A|=t\geq k$ and  $A=\{u_{1},u_{2},\ldots,u_{t}\}.$
Without loss of generality, assume that
$x_{u_{1}}\geq x_{u_{2}}\geq \cdots\geq x_{u_{t}}.$ Denote by $A'=\{u_{1},u_{2},\ldots,u_{k}\},$
$B'=\cap_{v\in A'}N_{G}(v)$ and $C'=V(G)\setminus (A'\cup B').$
By (\ref{upper and lower bounds of q}), Lemmas~\ref{property of cup of sets} and \ref{larger component derive larger degree} , we have
\begin{eqnarray*}
|B'|&=&\left|\bigcap_{v\in A'}N_{G}(v)  \right|  \geq \sum_{v\in A'}\left|N_{G}(v)\right|-(|A'|-1)\left|\bigcup_{v\in A'} N_{G}(v)\right|\\
&> & |A'|\left(\frac{4(n-2)}{3}-\frac{q}{3}\right)-(|A'|-1)n
>n-\frac{k(8+\epsilon n)}{3}>n- \epsilon' n.
\end{eqnarray*}

\begin{claim}\label{independent set}
There is a copy of $(k-1)K_{3}$ in $G$ satisfying that $B'\setminus V_{0}$ is an independent set, where   $V_{0}$ is the vertex set of $(k-1)K_{3}.$
\end{claim}
\begin{proofclaim}~{\bf \ref{independent set}.}
Assume for a contradiction that for each copy of $(k-1)K_{3},$
$B'\setminus V((k-1)K_{3})$ is not an independent set. Let $H$ be a copy of $(k-1)K_{3},$ $B''=B'\setminus V(H)$ and  $b_{1}b_{2}\in E(G[B'']).$ We have $A'\subseteq V(H).$ Otherwise, assume that $w\in A'\setminus V(H),$ by the definition of $B',$  $H\cup G[\{b_{1},b_{2},w\}]$ is a copy of  $kK_{3}$ in $G,$ a contradiction.

Since $|A'|=k$ and there are only $k-1$ disjoint triangles in $G,$ there is one triangle $\Delta$ in $H$ containing two vertices of $A',$ denoted by $u_{1},u_{2}.$ 
Then $(H\setminus V(\Delta))\cup G[\{u_{1},b_{1},b_{2}\}]$ is a copy of $(k-1)K_{3}$ in $G,$ denoted  by $H'.$
Obviously, $u_{2}\in A'\setminus V(H').$  Thus, $B'\setminus V(H')$ is an independent set as $G$ is $kK_{3}$-free. Now $G$ contains a copy of $(k-1)K_{3}$ satisfying that $B'\setminus V(H')$ is an independent set, 
 a contradiction.
\end{proofclaim}

Choose a copy  of $(k-1)K_{3},$ denoted by $H_{0},$ which has  the property in Claim~\ref{independent set}. Let $V_{0}=V(H_{0}),$ $A''=A'\setminus V_{0},$  $B''=B'\setminus V_{0}$ and   $C''=C'\setminus V_{0}.$ Obviously, $B''$
is an independent set and
$$|B''|\geq |B'|-|V_{0}|>n-\epsilon'n-3(k-1)>0.$$
In the sequel, we divide it into two cases.

{\bf Case~1.} $A''\neq \emptyset.$
Since $|B''|>0$ and $G$ is $kK_{3}$-free,
$A''$
is an independent set. 
Let $u\in A''$ with $x_{u}=\max_{v\in A''}\{x_{v}\}.$ Let $C_{0}''=N_{G}(u)\cap C''$ and $C_{1}''=C''\setminus C_{0}''.$
Similarly,
$B''\cup C_{0}''$ is an independent set.  
Since  
$|A''|+|C_{1}''|=n-|V_{0}|-(|B''|+|C_{0}''|)\leq n-|V_{0}|-|B''|< \epsilon 'n $, by Lemmas~\ref{symmetrize and adding edges} and \ref{the upper bound of q}, we have $n+\sqrt{(k-1)n}-2<q\leq q(K_{3(k-1)}\vee K_{|A''|+|C_{1}''|,|B''|+|C_{0}''|})<n+\frac{\sqrt{n}}{2},$
a contradiction.

{\bf Case~2.} $A''= \emptyset.$
Since $|A'|=k$ and there are only $k-1$ disjoint triangles in $G,$ there is one triangle $\Delta$ in $H_0$ containing two vertices of $A',$ denoted by $u_{1},u_{2}.$ Without loss of generality, we assume that $x_{u_{1}}\geq x_{u_{2}}.$
Let $C_{0}''=N_{G}(u_{1})\cap C''$ and $C_{1}''= C'' \setminus C_{0}''.$

If $e(B''\cup  C_{0}'')\geq 1,$
then there  is an edge in $G[B''\cup C_{0}''],$  denoted by  $b_{1}b_{2}.$ 
Thus $(H_0-V(\Delta))\cup G[\{u_{1},b_{1},b_{2}\}]$ is a copy of $(k-1)K_{3}$ in $G,$ denoted  by $H'.$
Obviously,  $u_{2}\in A'\setminus V(H')\neq \emptyset.$ Using a similar discussion as in Case~1, we derive a contradiction.

If $e(B''\cup  C_{0}''))= 0,$
then $B''\cup  C_{0}''$ is an independent set. Since $|C_{1}''|=n-|V_{0}|-(|B''|+|C_{0}''|)<\epsilon 'n,$ 
by  Lemmas~\ref{symmetrize and adding edges} and \ref{the upper bound of q}, we have $n+\sqrt{(k-1)n}-2<q\leq q(K_{3(k-1)-1}\vee K_{|C_{1}''|+1,|B''|+|C_{0}''|})<n+\frac{\sqrt{n}}{2},$
a contradiction.

Thus, $|A|\leq k-1.$
\end{proof}

By Lemma~\ref{the upper bound of size A}, $|A|\leq k-1.$
By the definition of  $B,$ (\ref{upper and lower bounds of q}), Lemmas~\ref{property of cup of sets} and  \ref{larger component derive larger degree}, we have
\begin{eqnarray}\label{the size of B}
|B|&=&\left|\bigcap_{v\in A}N_{G}(v)  \right|  \geq \sum_{v\in A}\left|N_{G}(v)\right|-(|A|-1)\left|\bigcup_{v\in A} N_{G}(v)\right|\nonumber\\
&> & |A|\left(\frac{4(n-2)}{3}-\frac{q}{3}\right)-(|A|-1)n
>n-\frac{(k-1)(8+\epsilon n)}{3}>n- \epsilon' n.
\end{eqnarray}

\begin{lemma}\label{extremalstructure}
$d_{G}(u)=n-1$ for each $u\in A.$ Furthermore, if $H$ is a copy of $(k-1)K_3$ in $G$, then  there are $|A|$ disjoint  copies of  $K_3$ in  $H,$ each of which contains exactly one vertex in $A$.
\end{lemma}
\begin{proof}
Firstly, we prove the following claims.

\begin{claim}\label{distribution of vertices of triangles}
If $H$ is a copy of $(k-1)K_3$ in $G$, then $A\subseteq V(H).$
\end{claim}
\begin{proofclaim}~{\bf \ref{distribution of vertices of triangles}.}
Suppose $H$ is a copy of $(k-1)K_3$ in $G$ with $A \not\subseteq V(H).$ Let $V_{0}=V(H),$ $A'=A\setminus V_{0},$ $B'=B\setminus V_{0}$ and $C'=C\setminus V_{0}$.  Obviously, $A'\neq \emptyset.$
By (\ref{the size of B}) and $n$ is sufficiently large, $|B'|>0.$  Furthermore, since $G$ is $kK_{3}$-free,
$A'$
is an independent set.
Let $u\in A'$ with $x_{u}=\max_{v\in A'}\{x_{v}\},$
$C_{0}'=N_{G}(u)\cap C'$ and $C_{1}'=C'\setminus C_{0}'.$ Similarly, $B'\cup C_{0}'$ is an independent set.
Since $|A'|+|C_{1}'|=n-|V_{0}|-(|B'|+|C_{0}'|)<\epsilon 'n,$
 $A'$ and $B'\cup C_{0}'$ are   independent sets,
by Lemmas~\ref{symmetrize and adding edges} and~\ref{the upper bound of q}, we have $n+\sqrt{(k-1)n}-2<q\leq q(K_{3(k-1)}\vee K_{|A'|+|C_{1}'|,|B'|+|C_{0}'|})<n+\frac{\sqrt{n}}{2},$
a contradiction.
\end{proofclaim}

\begin{claim}\label{C is empty}
$C=\emptyset.$
\end{claim}
\begin{proofclaim}~{\bf \ref{C is empty}.}
Assume for a contradiction that $C \neq \emptyset$. By the definition of $C,$ there are two non-adjacent vertices  $w \in C$ and $u \in A$ in $G.$ Let $G'$  be a graph obtained from $G$ by adding an  edge $uw.$ Since $G$ is extremal,  $G'$ contains a copy of  $kK_3$ with one triangle containing the edge $uw.$  Delete the triangle
containing $uw$ from the copy of $kK_3$, we obtain  a copy of  $(k-1)K_3$ in $G$ such that
$A \not\subseteq V((k-1)K_3)$, which contradicts  Claim~\ref{distribution of vertices of triangles}. 
\end{proofclaim}

\begin{claim}\label{A is a complete graph}
For each $u,w \in A$,  $uw \in E(G)$.
\end{claim}
\begin{proofclaim}~{\bf \ref{A is a complete graph}.}
 Suppose, for a contradiction, that there are two non-adjacent vertices $u, w \in A$.
Let $G'$  be a graph obtained from $G$ by adding an  edge $uw.$ Using a similar discussion as in Claim~\ref{C is empty}, we obtain  a contradiction.
\end{proofclaim}

By Claims~\ref{C is empty} and~\ref{A is a complete graph}, we have  $d_G(u) = n-1$ for each $u \in A$.

\begin{claim}\label{|A|K3}
If $H$ is a copy of $(k-1)K_3$ in $G$, then each triangle in $H$ contains at most one vertex in $A$.
\end{claim}
\begin{proofclaim}~{\bf \ref{|A|K3}.}
Suppose for a contradiction that $H$ is a copy of $(k-1)K_3$ in $G$ with one triangle  containing  two vertices of $A$, denoted by $u,w$. Let $B'=V(G)\setminus V(H).$  Then $B'$ is an independent set. Otherwise, suppose $st\in E(G[B'])$, by $d_G(u)=n-1$, deleting  the triangle containing $u,w$ from $H$ and adding a new triangle with the vertex set  $\{u,s,t\}$ to $H$, we obtain  a copy of $(k-1)K_3$ in $G$ without  containing $w$, which contradicts to Claim~\ref{distribution of vertices of triangles}.
By the definition of $B,$  Claims~\ref{C is empty} and~\ref{A is a complete graph},  $G$ is a subgraph of $K_{3(k-1)}\vee \overline{K}_{n-3(k-1)}.$
By
 $q(K_{3(k-1)}\vee \overline{K}_{n-3(k-1)})=\frac{n+6k-8+\sqrt{(n+6k-8)^2-
24(k-1)(3k-4)}}{2},$
we have $n+\sqrt{(k-1)n}-2<q\leq q(K_{3(k-1)}\vee \overline{K}_{n-3(k-1)})<n+6k-8,$ a contradiction as $n$ is sufficiently large.
\end{proofclaim}

Hence, if $H$ is a copy of $(k-1)K_3$ in $G,$ then   there are $|A|$ triangles in  $H,$ each of which contains exactly one vertex in $A$ by Claims~\ref{distribution of vertices of triangles} and \ref{|A|K3}.
\end{proof}

By Lemma~\ref{extremalstructure}, $d_{G}(u)=n-1$ for each $u\in A.$  According to the proof of Lemma~\ref{maximum degree is n-1}, each vertex in $A$ satisfies equality~(\ref{equality}). Then we have the following corollary.
\begin{corollary}\label{the number of independent edges in B}
$G[N_{G}(u)]$ contains at least
$\left\lfloor\frac{\sqrt{(k-1)n}}{2}\right\rfloor$ independent edges for each $u\in A.$
\end{corollary}

\begin{lemma}\label{degreecondition}
 For each  $u \in B$,  $d_{G}(u) \leq n - 1-\binom{72k^2-9k-2}{3k-1}$.
\end{lemma}
\begin{proof}
Let $c=\binom{72k^2-9k-2}{3k-1}+1.$
Since $x_{u^*} = 1$, by Lemma~\ref{maximum component implies maximum degree}, $d_{G}(u^*) =n-1.$
By $Q(G)\mathbf{x} = q\mathbf{x}$, we have
$$
q x_{u*} = (n-1) x_{u*}+ \sum_{v \in N_{G}(u^*)} x_v, ~~~~~~~
q x_u = d_{G}(u) x_u + \sum_{v \in N_{G}(u)} x_v \quad \text{for}~  u \in B.
$$
For any  $u\in B,$  $d_{G}(u)\leq n-2$ by Lemma~\ref{maximum component implies maximum degree}.
Thus for any  $u\in B,$
\begin{align*}
x_u &= 1 - \frac{n-2 - d_{G}(u) + \sum_{v \notin N_{G}(u)} x_v}{q - d_{G}(u)} \geq 1 - \frac{(n-2 - d_{G}(u)) + (n- d_{G}(u))}{q - n + 2}.
\end{align*}
If $d_{G}(u) > n - c$, then
\begin{align*}
x_u &> 1 - \frac{(n-2 - n + c) + (n - n + c)}{q - n + 2} = 1 - \frac{2c - 2}{q - n + 2} > \frac{3}{4},
\end{align*}
where the last inequality holds as $q> n+\sqrt{(k-1)n}-2$ and $n$ is sufficiently large,  a contradiction.
\end{proof}

\begin{lemma}\label{the eaxct size of A}
 $|A|= k-1.$
\end{lemma}
\begin{proof}
By Lemma 3.3,  $|A| \leq k-1$.  Suppose for a contradiction that $|A| \leq k-2$. Let $v^* \in B$ with
$
x_{v^*} = \max_{w \in B} \{x_w\}
$
and  $B_{1}= N_B(v^*)$. Let $\nu(B_{1})$ be the matching number of $G[B_{1}]$ and  $\{u_1u_1', \dots, u_su_s'\}$ be a maximum matching of $G[B_{1}]$, where $s = \nu(B_{1})$. Let $V_1 = \{u_1,u_1',\dots,u_s,u_s'\}$ and $B_{2}= B \setminus (B_{1} \cup \{v^*\}).$

\textbf{Case 1.} If $s \geq 3k$, then $d_G(v^*) = n-1$. Otherwise, combining with $C=\emptyset$ by Lemma~\ref{extremalstructure},
there exists a vertex $w \in B_{2}$ satisfying that $w$ is not adjacent to $v^{*}$. Let $G'$ be a graph obtained from $G$ by adding  the edge $v^*w$. Then $G'$ contains a copy of $kK_3$ with one triangle containing the edge $v^*w$, which implies that $G$ contains  a copy of $(k-1)K_3$ without containing $v^*$, denoted by $H.$ Since $s \geq 3k$, there exist two adjacent vertices $u_{i},u_{i}'\notin V(H).$ 
Thus, $H\cup G[\{v^*,u_{i},u_{i}'\}]$ is a copy of $kK_3$ in $G$, a contradiction. By Lemma~\ref{maximum component implies maximum degree},  $v^*\in A,$ a contradiction.

\textbf{Case 2.} If $s = 0$, then $B_{1}$ is an independent set. Let $G'$ be a graph obtained from $G$ by symmetrizing   $v$ to $v^*$ for each $v \in B_2 $. Clearly, $G'=K_{|A|} \vee K_{|B_{1}|, n-|A|-|B_{1}|}$.
Then $G'$ is $kK_3$-free as $G'$  has at most $|A|$ independent triangles.
Thus,
\[
q(G) \leq q(G') < q(K_{k-1} \vee K_{|B_{1}|, n-k+1-|B_{1}|}) \leq q(K_{k-1} \vee T_2(n-k+1)),
\]
where the first and last inequalities follow from Lemmas~\ref{symmetrize and adding edges} and \ref{JinLe-2}, respectively, and the second inequality holds as $G'$ is a proper subgraph of $K_{k-1} \vee K_{|B_{1}|, n-k+1-|B_{1}|},$ 
 a contradiction.

\textbf{Case 3.} If $1 \leq s \leq k-1-|A|$, then let $G'$ be a graph obtained from $G$ by symmetrizing   $v$ to $v^*$ for each $v \in B_2 $. Clearly, $G'=K_{|A|} \vee \overline{K}_{n-|A|-|B_{1}|} \vee G[B_{1}].$ Then $G'$ is $kK_3$-free. Otherwise,  there exists a copy of $(k-|A|)K_{3}$ in $G'[B]$ as there are at most $|A|$ disjoint copies of $K_{3}$ containing a vertex of $A.$ According to the definition of $G'$, 
there are $(k-|A|)$ independent edges in  $G[B_{1}]$, a contradiction to $s \leq k-1-|A|.$ By Lemma~\ref{symmetrize and adding edges} and $G$ is extremal,
\[
q(G) \leq q(G') \leq q(G),
\]
which implies that $G = G'=K_{|A|} \vee \overline{K}_{n-|A|-|B_{1}|} \vee G[B_{1}].$ Since $G$ contains a copy of $(k-1)K_{3},$ it follows from Lemma~\ref{extremalstructure} that $s= k-1-|A|$.

Let   $w^* \in V_1$ with
$
x_{w^*} = \max_{w \in V_1} \{x_w\}.
$
Without loss of generality, assume $w^* = u_1$.
Let $G''$ be a graph obtained from $G$ by  deleting  all edges incident to $u_1,$  $u_1'$ from $G[B_{1}]$, respectively, and adding all possible edges between $u_1$ and all vertices of $B_{1}\setminus V_1$.  Since $G = K_{|A|} \vee \overline{K}_{n-|A|-|B_{1}|} \vee G[B_{1}]$,
\[
q x_{u_1} = d_{G}(u_1)x_{u_1} + \sum_{w \in N_{G}(u_1)} x_w, \quad q x_v = d_{G}(v)x_v + \sum_{w \in N_{G}(v)} x_w \text{ for } v \in B_{1}\setminus V_1.
\]
Then
\[
q x_{u_1} \geq (n-|B_{1}|)x_{u_1} + \sum_{w \in N_{G}(u_1)} x_w > (n-|B_{1}|)x_{u_1} + 1,
\]
which implies that
\[
x_{u_1} > \frac{1}{q - n + |B_{1}|}.
\]
Moreover, for $v \in B_{1}\setminus V_1,$
\[
\begin{aligned}
q(x_{u_1} - x_v) &= d_{G}(u_1)x_{u_1} - d_{G}(v)x_v + \sum_{w \in N_{G}(u_1)} x_w - \sum_{w \in N_{G}(v)} x_w \\
&\geq (n-|B_{1}|)x_{u_1} - (n-|B_{1}|+2s)x_v  - 2s x_{u_1}
\geq (n-|B_{1}|)(x_{u_1} - x_v) - 4s,
\end{aligned}
\]
which implies
\[
x_v \leq x_{u_1} + \frac{4s}{q-n+|B_{1}|} < (4s+1)x_{u_1}.
\]

Since $u_1'$ is an isolated vertex in $G''[B_{1}]$ and   $G''[B_{1}] - \{u_1, u_1'\} = G[B_{1}] - \{u_1, u_1'\},$  $\nu(G''[B_{1}])=s=k-1-|A|$. So $G''$ is $kK_3$-free.
Since $u_1,u_1' \in B$, by Lemma~\ref{degreecondition}, we have
\[
d_G(u_1) \leq n -1- {72k^2 - 9k - 2\choose 3k-1},~ d_G(u_1') \leq n -1- {72k^2 - 9k - 2\choose 3k-1}.
\]
Thus there are at least
\[
{72k^2 - 9k - 2\choose 3k-1} - 2(k-1-|A|)
\]
vertices of $B_{1}\setminus V_1$ not adjacent to $u_1,u_1',$ respectively.
Since $\{u_1u_1',\dots,u_su_s'\}$ is a maximum matching of $G[B_{1}]$,
we have $|N_{B_{1}\setminus V_1}(u_1)| \leq 1$ or $|N_{B_{1}\setminus V_1}(u_1')| \leq 1$.
 If $|N_{B_{1}\setminus V_1}(u_1)| \leq 1,$ then
\begin{eqnarray*}
&&\mathbf{x}^{\mathrm{T}}(Q(G'')-Q(G))\mathbf{x}\\
&=&\left(\sum_{\substack{ v\in B_{1}\setminus V_{1}\\ vu_{1}'\in E(G)}} (x_{v}+x_{u_{1}})^2+\sum_{\substack{ v\in B_{1}\setminus V_{1}\\ vu_{1}'\notin E(G)}} (x_{v}+x_{u_{1}})^2-\sum_{\substack{ v\in B_{1}\setminus V_{1}\\ vu_{1}\in E(G)}} (x_{v}+x_{u_{1}})^2\right)-\sum_{\substack{ v\in B_{1}\setminus V_{1}\\ vu_{1}'\in E(G)}} (x_{v}+x_{u_{1}'})^2\\
&&-\sum_{\substack{ v\in  V_{1}\\ vu_{1}\in E(G)}} (x_{v}+x_{u_{1}})^2-\sum_{\substack{ v\in  V_{1}\setminus\{
u_{1}\}\\ vu_{1}'\in E(G)}} (x_{v}+x_{u_{1}'})^2\\
&\geq &
\sum_{\substack{ v\in B_{1}\setminus V_{1}\\ vu_{1}'\notin E(G)}} (x_{v}+x_{u_{1}})^2-\sum_{\substack{ v\in B_{1}\setminus V_{1}\\ vu_{1}\in E(G)}} (x_{v}+x_{u_{1}})^2-\sum_{\substack{ v\in  V_{1}\\ vu_{1}\in E(G)}} (x_{v}+x_{u_{1}})^2-\sum_{\substack{ v\in  V_{1}\setminus\{
u_{1}\}\\ vu_{1}'\in E(G)}} (x_{v}+x_{u_{1}'})^2\\
&> &\left(\binom{72k^2-9k-2}{3k-1} - 2(k-1-|A|)\right)x_{u_1}^2 - (4s+2)^2x_{u_1}^2 - 2\cdot2s\cdot4 x_{u_1}^2.
\end{eqnarray*}
If $|N_{B_{1}\setminus V_1}(u_1')| \leq 1,$ then
\begin{eqnarray*}
&&\mathbf{x}^{\mathrm{T}}(Q(G'')-Q(G))\mathbf{x}\\
&=&\sum_{\substack{ v\in B_{1}\setminus V_{1}\\ vu_{1}\notin E(G)}} (x_{v}+x_{u_{1}})^2-\sum_{\substack{ v\in B_{1}\setminus V_{1}\\ vu_{1}'\in E(G)}} (x_{v}+x_{u_{1}'})^2-\sum_{\substack{ v\in  V_{1}\\ vu_{1}\in E(G)}} (x_{v}+x_{u_{1}})^2-\sum_{\substack{ v\in  V_{1}\setminus\{
u_{1}\}\\ vu_{1}'\in E(G)}} (x_{v}+x_{u_{1}'})^2\\
&> &
\left(\binom{72k^2-9k-2}{3k-1} - 2(k-1-|A|)\right)x_{u_1}^2 - (4s+2)^2x_{u_1}^2 - 2\cdot2s\cdot4 x_{u_1}^2.
\end{eqnarray*}
In both situations, since $|A| \geq 1$ and $s\leq k-1-|A|<k-1,$ we have
\[
\begin{aligned}
0\geq (q(G'')-q)\mathbf{x}^{\mathrm{T}}\mathbf{x} &\geq  \mathbf{x}^{\mathrm{T}}(Q(G'')-Q(G))\mathbf{x}
 >\left(\binom{72k^2-9k-2}{3k-1} - 2k - 16k^2 - 16k\right)x_{u_1}^2 > 0,
\end{aligned}
\]
a contradiction. 

\textbf{Case 4.}  $k-|A| \leq s < 3k$.
If every copy of $(k-1)K_3$ in $G$ contains $v^*$,
then $d_{G}(v^*) = n-1.$
          Otherwise,
          combining with $C=\emptyset$ by Lemma~\ref{extremalstructure},
there exists a vertex $w \in B_{2}$ satisfying that $w$ is not adjacent to $v^{*}$.
Let $G'$ be a graph obtained from $G$ by adding  the edge $v^*w$. Then $G'$ contains a copy of $kK_3$ with one triangle containing the edge $v^*w$, which implies that $G$ contains  a copy of $(k-1)K_3$ without containing $v^*$,  contradicting the assumption. Thus $d_{G}(v^*) = n-1.$ However, by Lemma~\ref{maximum component implies maximum degree},  $v^*\in A,$ a contradiction.
        Thus there exists a copy of $(k-1)K_3$ in $G$  without containing  $v^*$.
        Firstly, we prove the following claim.
        \begin{claim}\label{two different independent edges}
      If $H$ is  a copy of $(k-1)K_{3}$ without containing  $v^*,$ then
        there is one triangle in  $H$ containing two vertices incident with two distinct  edges in $\{u_1u_1',\dots,u_su_s'\}$.
        \end{claim}
 \begin{proofclaim}~{\bf \ref{two different independent edges}.}
 Firstly,  we prove that $ \{u_i,u_i'\}\cap V(H)\neq \emptyset$ for each $1\leq i\leq s.$ Otherwise, without loss of generality, suppose that  $\{u_1,u_1'\}\cap V(H)= \emptyset.$ Since $v^{*}\notin V(H),$ $H\cup G[\{v^{*},u_1,u_1'\}]$ is a copy of $kK_{3}$ in $G,$ a contradiction.
 Hence, there exists a triangle of $H$ that intersects the set  $\{u_i,u_i'\}$ nontrivially for each edge  $ u_iu_i'\in \{u_1u_1',\dots,u_su_s'\}.$
   Next we prove that  for each edge  $ u_iu_i'\in \{u_1u_1',\dots,u_su_s'\},$
   among the triangles of $H$ that meet $\{u_i,u_i'\}$, at least one has all its vertices in $B$.



Fix an arbitrary index $i\in\{1,\ldots,s\}$.
Suppose, for a contradiction, that every triangle of $H$
meeting $\{u_i,u_i'\}$ also meets $A$.
   Denote  the triangles of $H$ satisfying $\{u_i,u_i'\}\cap V(\Delta_{i})\neq \emptyset$ by $\Delta_{i}^{1},\cdots ,\Delta_{i}^{t_{i}},$ where $1\leq t_{i} \leq 2.$ By the assumption,  there exists $w_{i}^{j}\in V(\Delta_{i}^{j})\cap A$ for $1\leq j\leq t_{i}.$
  Obviously, $w_{i}^{j}\notin \{u_{i},u_{i}'\}\subseteq B$ for $1\leq j\leq t_{i}.$ By Corollary~\ref{the number of independent edges in B}, for $1\leq j\leq t_{i},$  $G[N_{G}(w_{i}^{j})]$ has at least $3k+3$ independent edges, one of which, denoted by   $p_{i}^{j}q_{i}^{j},$ satisfies $\{p_{i}^{j},q_{i}^{j}\}\subseteq B,$  $\{p_{i}^{j},q_{i}^{j}\}\cap (V(H)\cup \{u_{i},u_{i}',v^{*}\})=\emptyset$ and the sets $\{p_j^i,q_j^i\}$, $1\le j\le t_i$, are pairwise disjoint. Thus, $(H-  \cup_{1\leq j\leq t_{i}}V(\Delta_{i}^{j}))\cup (\cup_{1\leq j\leq t_{i}} G[\{w_{i}^{j},p_{i}^{j},q_{i}^{j}\}])$ is a copy of $(k-1)K_{3}$ in $G$ without containing any vertex in $\{v^{*},u_{i},u_{i}'\},$  denoted by $H'.$ Since $v^{*}\notin V(H),$  $H'\cup G[\{v^{*},u_{i},u_{i}'\}]$ is a copy of $kK_{3}$ in $G,$ a contradiction.

  Now we prove the claim by way of contradiction.
  Assume  that for each triangle  $\Delta $ in $H,$  there exists some $1\leq i\leq s$ such that $$V_{1}\cap V(\Delta) =\{u_1,u_1',\dots,u_s,u_s'\} \cap V(\Delta) \subseteq \{u_{i},u_{i}'\}.$$   
  Recall that for each edge  $ u_iu_i'\in \{u_1u_1',\dots,u_su_s'\},$  there is one triangle $\Delta_{i}$ of $H$ satisfying that $ \{u_i,u_i'\}\cap V(\Delta_{i})\neq \emptyset$ and
   $V(\Delta_{i})\subseteq B.$
   There is a one-to-one correspondence between the set $\{u_1u_1',\dots,u_su_s'\}$ and $s$ distinct triangles  in $H$  and also in $G[B].$
   By Lemma~\ref{extremalstructure},  there are exactly $|A|$ disjoint copies of $K_{3}$ in  $H,$ each of which contains exactly one vertex in $A.$  Thus, there are $s+|A|\geq k-|A|+|A|=k$ disjoint copies of $K_{3}$ in $G,$  a contradiction.
 \end{proofclaim}

Let $G''$ be a graph obtained from $G$ by deleting  all edges in $G[\{v^*\} \cup V_1]$.
Then  every copy of  $(k-1)K_3$ in $G''$ contains $v^{*}$. Otherwise, assume that $G''$ contains a copy of $(k-1)K_3$ without containing $v^*,$ denoted by $H.$  By Claim~\ref{two different independent edges}, there is one triangle in  $H$ containing one edge in $G[V_1]$.  However,  $G''$ contains no edges in $G[V_1]$, a contradiction.

By Lemma~\ref{degreecondition},
    \[
    |B_{2}| \geq {72k^2 - 9k - 2\choose 3k-1}.
    \]
    Since $G$ is $kK_3$-free, the clique number of $G[B_{2}]$ is less than $3k$. By Lemma 2.1, the independence number of $G[B_{2}]$ satisfies
    \[
    \alpha(G[B_{2}]) \geq 4\binom{6k}{2}.
    \]
    Let $I$ be a maximum independent set of $G[B_{2}]$. So $|I| \geq 4\binom{6k}{2}$. Let $G'''$ be a graph obtained from $G''$ by adding all possible edges between $v^*$ and $I$. Then $G'''$ is $kK_3$-free. Otherwise, $G'''$ contains a copy of $kK_3$, denoted by $H'$, which must contain a triangle with vertices $v^*,w,w'$, where $w \in N_B(v^*)\cup A$ and $w' \in I$. Thus, $H' - \{v^*,w,w'\}$ is   a copy of $(k-1)K_3$ in $G''$, which implies that $G''$ contains a copy of $(k-1)K_3$ not containing $v^*,$ a contradiction to the fact that each copy of  $(k-1)K_3$ in $G''$ containing $v^{*}$.
    However,
    \[
    0\geq (q(G''')-q)\mathbf{x}^{\mathrm{T}}\mathbf{x} \geq\mathbf{x}^{\mathrm{T}}(Q(G''')-Q(G))\mathbf{x} \geq |I|x_{v^*}^2 - 4\binom{6k-1}{2}x_{v^*}^2 > 0,
    \]
    a contradiction.

So $|A| = k-1$.
\end{proof}

\begin{lemma}\label{B is a complete biaprtite graph}
$G[B]$ is a complete bipartite graph.
\end{lemma}
\begin{proof}
Firstly we prove that $G[B]$ is  $K_{3}$-free.
Assume for a contradiction that $G[B]$ contains a  triangle, denoted by $\Delta.$  By Corollary~\ref{the number of independent edges in B}, for each $u\in A,$ there are at least $2k$ independent edges in $G[N_{G}(u)]$.
By Lemma~\ref{the eaxct size of A} and $|V(\Delta)|=3$, there are at most $k+1$ independent edges in  $G[N_{G}(u)]$ which contains a vertex in $A\cup V(\Delta).$
Thus, there are at least $k-1$ independent edges in  $G[B\setminus V(\Delta)],$ which implies that there is a copy of $(k-1)K_{3}$ in $G- V(\Delta),$ denoted by $H.$ Hence, $H\cup \Delta $ is a copy of $kK_{3}$ in $G,$ a contradiction.

Next we prove that $G[B]$ is a complete bipartite graph.
Choose $v^{*}\in V(B)$ with
$x_{v^{*}}= \max\{x_{u} : u \in V(B)\}.$
Let $B_{1}=N_{B}(v^{*})$ and $B_{2}=B\setminus (B_{1}\cup \{v^{*}\})$. Then $V(G) = A\cup B_{1}\cup B_{2} \cup \{v^{*}\}$. Clearly, $B_{1}\neq \emptyset$. 
Otherwise, $d_{B}(v^{*})=0.$ Since $x_{u}=1$ for each $u\in A,$ $qx_{v^{*}}=|A|x_{v^{*}}+|A|,$ which implies that $x_{v^{*}}=\frac{|A|}{q-|A|}.$ If $d_{B}(v)>0$ for some $v\in B,$ then
$qx_{v}=(|A|+d_{B}(v))x_{v}+|A|+\sum_{z\in N_{B}(v)}x_{z},$ which implies that $x_{v}>\frac{|A|}{q-|A|}=x_{v^{*}},$ a contradiction.
Thus, $B$ is an independent set, a contradiction to Corollary~\ref{the number of independent edges in B}.
Since $G[B]$ is $K_{3}$-free, $B_{1}$ is an independent set.

By Lemma~\ref{degreecondition}, $B_{2}\neq \emptyset.$
Let $G'$ be a graph obtained from $G$ by symmetrizing all vertices in $B_{2}$ to $v^{*}.$
Clearly, $G'= K_{k-1} \vee K_{|B_{1}|,|B_{2}\cup \{v^{*}\}|}.$
 Since $K_{|B_{1}|,|B_{2}\cup \{v^{*}\}|}$ is $K_{3}$-free,  $G'$ is $kK_{3}$-free by Lemma~\ref{join}.
Furthermore, by $G$ is extremal and  Lemma~\ref{symmetrize and adding edges}, we have $q\geq q(G') \geq  q,$
equality holds if and only if $G = G'$. Hence,  $G[B]$ is a complete bipartite graph.
\end{proof}

According to Lemmas~\ref{extremalstructure}, \ref{the eaxct size of A} and  \ref{B is a complete biaprtite graph}, 
we have
$G=K_{k-1}\vee G[B],$ where $K_{k-1}$ corresponds to the vertices in $A$ and $G[B]$ is a complete bipartite graph.
Since $G$ is extremal, by Lemma~\ref{JinLe-2}, we have
$G = K_{k-1} \vee T_{2}(n-k+1)$.
\section{Conclusions}
Recall that an induced matching of size $k$ in a graph $F$ is a set of $k$ edges such that no two of them intersect in a common vertex or are joined by an edge of $F$; that is, an induced matching is a matching that forms an induced subgraph.

\begin{definition}\label{color-k-critical}
Let $F$ be a graph with $\chi(F)=r+1$ and  $k\geq 1$
be an integer. We call $F$ \emph{color-$k$-critical} if
the following conditions hold:
\begin{enumerate}
    \item[(i)]
    For every $S\subseteq V(F)$ with $|S|=k-1$, we have
    \[
        \chi(F-S)=r+1.
    \]

    \item[(ii)]
    There exists a partition
    \[
        V(F)=U_1\cup\cdots\cup U_r
    \]
    such that $F[U_1]$ consists of $k$ pairwise disjoint
    edges and possibly some isolated vertices, while
    $U_2,\ldots,U_r$ are independent sets.
\end{enumerate}
\end{definition}

 It is worth highlighting that this family includes a wide range of graphs, such as color-critical graphs, the Petersen graph, Kneser graphs $KG(n,2)$, and disjoint unions of cliques.
In 1974, Simonovits \cite{Simonovits1974} determined the unique extremal graph for color-$k$-critical graphs.

\begin{theorem}[Simonovits \cite{Simonovits1974}]
Let $k \ge 1$, $r \ge 2$ and $F$ be a color-$k$-critical graph with $\chi(F) = r + 1$. If $n$ is sufficiently large and $G$ is an $n$-vertex $F$-free graph, then
\[
ex(n,F)= e(K_{k-1}\vee  T_r(n-k+1)),
\]
and the unique edge extremal graph is $K_{k-1}\vee  T_r(n-k+1)$.
\end{theorem}

Note that a color-$1$-critical graph is precisely a color-critical graph. Let $F$ be a color-critical graph with $\chi(F)=r+1\geq 4$. In~\cite{Zheng  Li Li 2026}, Zheng,  Li and  Li proved that for sufficiently large  $n,$ if  $G$ is an $n$-vertex $F$-free graph, then $T_{r}(n)$ is the unique signless Laplacian spectral extremal graph. In the meantime, for $k\geq 2$ and  a color-$k$-critical graph $F$ with $\chi(F) = r + 1\geq 4,$   they conjectured that $K_{k-1}\vee T_{r}(n-k+1)$ is  the unique graph maximizing $q(G)$ among all $n$-vertex $F$-free graphs if $n$ is sufficiently large. Recently, this conjecture was confirmed  by Chen, et al.~\cite{Chen-Jin-Zhang-Zheng2026}.

By Theorem~\ref{main result}, $K_{k-1}\vee T_{2}(n-k+1)$ is the unique signless Laplacian spectral extremal graph for
an $n$-vertex $kK_{3}$-free graph if $n$ is sufficiently large. Note that, for $k \geq 2$, $kK_{3}$ is also a color-$k$-critical graph with $\chi(kK_{3}) = 3$. Motivated by this observation, we propose the following conjecture for color-$k$-critical graphs $F$ with $\chi(F) = 3.$
\begin{conjecture}\label{conjecture}
Let $k \geq 2$ and $F$ be a color-$k$-critical graph with $\chi(F) = 3$. If $n$ is sufficiently large and $G$ is an $n$-vertex $F$-free graph, then
\[
q(G) \leq q(K_{k-1}\vee T_{2}(n-k+1)).
\]
Moreover,  equality holds if and only if $G = K_{k-1}\vee T_{2}(n-k+1)$.
\end{conjecture}

\medskip
\noindent\textbf{Acknowledgements.}
The main result of this paper was obtained independently of the work of
Jing Zeng \cite{Zeng2026}, who established the same extremal characterization
with the explicit bound $n \ge 28k-17$.
Our proof follows a different approach, combining a partition of the vertex
set according to Perron vector entries with a degree-based upper bound
for the signless Laplacian spectral radius.
Our manuscript was completed in mid-May 2026.

\end{document}